\documentclass[leqno]{article}
\usepackage{geometry}
\usepackage{graphicx}	
\usepackage{mathtools}
\usepackage{amsfonts,amssymb,mathrsfs,amscd,amsmath,amsthm}
\usepackage{verbatim}
\usepackage{url}

\def\thtext#1{
  \catcode`@=11
  \gdef\@thmcountersep{. #1}
  \catcode`@=12
}

\def\threst{
  \catcode`@=11
  \gdef\@thmcountersep{.}
  \catcode`@=12
}

 \catcode`@=11
 \def\.{.\spacefactor\@m}
 \catcode`@=12

\theoremstyle{plain}
\newtheorem{thm}{Theorem}[section]
\newtheorem{prop}{Proposition}[section]
\newtheorem{cor}[prop]{Corollary}

\theoremstyle{definition}
\newtheorem{dfn}[prop]{Definition}
\newtheorem{rk}[prop]{Remark}

\newtheorem{notation}[prop]{Notation}

\newcommand{\cM}{\mathcal{M}}
\newcommand{\cP}{\mathcal{P}}
\newcommand{\cR}{\mathcal{R}}

\newcommand{\R}{\mathbb{R}}

\newcommand{\dl}{\delta}
\newcommand{\D}{\Delta}
\newcommand{\e}{\varepsilon}

\renewcommand{\l}{\lambda}
\renewcommand{\r}{\rho}
\newcommand{\s}{\sigma}
\renewcommand{\t}{\tau}
\renewcommand{\v}{\varphi}

\newcommand{\diam}{\operatorname{diam}}
\newcommand{\dis}{\operatorname{dis}}

\newcommand{\opt}{{\operatorname{opt}}}

\newcommand{\0}{\emptyset}

\def\rom#1{\emph{#1}}
\def\({\rom(}
\def\){\rom)}

\renewcommand{\ss}{\subset}
\newcommand{\x}{\times}

\begin{document}
\title{Local Structure of Gromov--Hausdorff Space near Finite Metric Spaces in General Position}
\author{Alexander O. Ivanov, Alexey A. Tuzhilin}
\maketitle

\begin{abstract}
We investigate the local structure of the space $\cM$ consisting of isometry classes of compact metric spaces, endowed with the Gromov--Hausdorff metric. We consider finite metric spaces of the same cardinality and suppose that these spaces are in general position, i.e., all nonzero distances in each of the spaces are distinct, and all triangle inequalities are strict. We show that sufficiently small balls in $\cM$ centered at these spaces and having the same radii are isometric. As consequen\-ces, we prove that the cones over such spaces (with the vertices at one-point space) are isometrical; the isometry group of each sufficiently small ball centered at a general position $n$-points space, $n\ge3$, contains a subgroup isomorphic to the group $S_n$ of permutations of a set containing $n$ points.
\end{abstract}

\section{Introduction}
\markright{\thesection.~Introduction}
Denote by $\cM$ the space of all compact metric spaces (considered up to isometry), endowed with the Gromov--Hausdorff metric. Such $\cM$ is usually called the space of Gromov--Hausdorff. We are interested in isometries related to $\cM$. In~\cite{IvaIliadisTuzIsom} we consider finite metric spaces $M$ such that all nonzero distances between their points are distinct, and all triangle inequalities are strict. Such $M$ we call the spaces in general position. It is shown in~\cite{IvaIliadisTuzIsom} that for any $n$-points space $M$ in general position, each its sufficiently small neighborhood in the space $\cM_n\ss\cM$ consisting of all spaces with at most $n$ points, is isometric to an open subset of $\R^k_\infty$, $k=n(n-1)/2$, where $\R^k_\infty$ denotes the arithmetic space $\R^k$ with metric generated by the norm $\bigl\|(x_1,\ldots,x_k)\bigr\|_\infty=\max_i|x_i|$. In particularly, this implies that the balls in $\cM_n$ with sufficiently small equal radii and centered at spaces in general position are isometric.

In the present paper we enhance this result. Namely, we show that for $n$-points spaces in general position, their round neighborhoods in \textbf{entire $\cM$}, having sufficiently small equal radii, are isometric. We construct these isometries explicitly. As a corollary, we get that the cones over such spaces are isometric to each other as well; here by the cone $CZ$ over a set $Z\ss\cM$ we mean  the set $\{\l X\in\cM:X\in Z,\,\l>0\}\cup\{\D_1\}$, where $\l X$ denotes the metric space obtained from  $X\in\cM$ by multiplying all its distances by $\l$, and $\D_1$ is one-point metric space.

Besides that, we show that small round neighborhoods of $n$-points spaces in general position, $n\ge 3$, have rather rich symmetry group, namely, this group contains a subgroup isomorphic to the group $S_n$ of all permutations of an $n$-points set.

\section{Preliminaries}
\markright{\thesection.~Preliminaries}
Let $X$ be an arbitrary metric space. The distance between its points $x$ and $y$ is denoted by $|xy|$. For any point $x\in X$ and a real number $r>0$ we denote by $U_r(x)$ the open ball of radius $r$ centered at $x$; for each nonempty $A\ss X$ and a real number $r>0$ we put $U_r(A)=\cup_{a\in A}U_r(a)$.

For nonempty $A,\,B\ss X$ let
$$
d_H(A,B)=\inf\bigl\{r>0:A\ss U_r(B)\&B\ss U_r(A)\bigr\}.
$$
The value just introduced is called the \emph{Hausdorff distance between $A$ and $B$}. It is well-known~\cite{BurBurIva} that the Hausdorff distance, being restricted to the set of all nonempty closed bounded subsets of $X$, is a metric.

Let $X$ and $Y$ be metric spaces. A triple $(X',Y',Z)$ consisting of a metric space $Z$, and its subsets $X'$ and $Y'$ isometric $X$ and $Y$, respectively, is called a \emph{realization of the pair $(X,Y)$}. \emph{The Gromov--Hausdorff distance $d_{GH}(X,Y)$ between $X$ and $Y$} is the infimum of real numbers $r$ for which there exist realizations $(X',Y',Z)$ of $(X,Y)$ with $d_H(X',Y')\le r$. It is well-known~\cite{BurBurIva} that the function $d_{GH}$, being restricted to the set $\cM$ of all isometry classes of compact metric spaces, forms a metric.

Recall that a \emph{relation\/} between sets $X$ and $Y$ is a subset of the Cartesian product $X\x Y$. We denote by $\cP(X,Y)$ the set of all nonempty relations between $X$ and $Y$. If $\pi_X\:X\x Y\to X$ and $\pi_Y\:X\x Y\to Y$ are the canonical projections, i.e., $\pi_X(x,y)=x$ and $\pi_Y(x,y)=y$, then the restrictions of them to each relation $\s\in\cP(X,Y)$ will be denoted in the same way.

Let us interpret each relation $\s\in\cP(X,Y)$ as a multivalued mapping, which domain may be less than the entire $X$. Then, similarly with the practice in the Theory of Mappings, for each $x\in X$ it is defined its image $\s(x)=\{y\in Y\mid(x,y)\in\s\}$; for each $y\in Y$ it is defined its preimage $\s^{-1}(y)=\{x\in X\mid(x,y)\in\s\}$; for each $A\ss X$ and $B\ss Y$ we also have $\s(A)=\cup_{x\in A}\s(x)$ and $\s^{-1}(B)=\cup_{y\in B}\s^{-1}(y)$.

A relation $R$ between $X$ and $Y$ is called a \emph{correspondence}, if the restrictions of the canonical projections $\pi_X$ and $\pi_Y$ onto $R$ are surjective. We denote by $\cR(X,Y)$ the set of all correspondences between $X$ and $Y$.

Let $X$ and $Y$ be metric spaces, then for each relation $\s\in\cP(X,Y)$ it is defined its \emph{distortion}
$$
\dis\s=\sup\Bigl\{\bigl||xx'|-|yy'|\bigr|: (x,y)\in\s,\ (x',y')\in\s\Bigr\}.
$$

The next result is well-known.

\begin{prop}[\cite{BurBurIva}]\label{th:GH-metri-and-relations}
For any metric spaces $X$ and $Y$ we have
$$
d_{GH}(X,Y)=\frac12\inf\bigl\{\dis R:R\in\cR(X,Y)\bigr\}.
$$
\end{prop}

\begin{dfn}
A correspondence $R\in\cR(X,Y)$ is called \emph{optimal}, if $d_{GH}(X,Y)=\frac12\dis R$. We denote by $\cR_\opt(X,Y)$ the set of all optimal correspondences between $X$ and $Y$.
\end{dfn}

If $X$ and $Y$ are finite metric spaces, then the set $\cR(X,Y)$ is finite, thus, in this case, there exists an optimal correspondence $R\in\cR(X,Y)$.

\begin{prop}[\cite{IvaIliTuz}, \cite{Memoli}, \cite{blog}]\label{prop:optimal-cor}
For any $X,\,Y\in\cM$ we have $\cR_\opt(X,Y)\ne\0$.
\end{prop}

For a metric space $X$, we denote by $\diam X$ its \emph{diameter $\sup\bigl\{|xx'|:x,x'\in X\bigr\}$}. For any real number $\l>0$, we write $\l X$ for the metric space obtained from $X$ by multiplying all its distances by $\l$.

The next result is well-know~\cite{BurBurIva}.

\begin{prop}\label{prop:elem_prop}
For any metric spaces $X$ and $Y$ we have
$$
d_{GH}(\l X,\l Y)=\l\,d_{GH}(X,Y).
$$
\end{prop}

To conclude the present section, we introduce a few more necessary notions and notations.

By $\D_1$ we denote one-point metric space. For $Z\ss\cM$ we define the \emph{cone $CZ\ss\cM$ over $Z$} as the set $\{\l X:X\in Z,\,\l>0\}\cup\{\D_1\}$.

For concrete calculations related to estimations of distortions, the following notations turned out to be useful, and they were intensively used in~\cite{IvaTuzSimpDist}.

\begin{notation}
For nonempty subsets $A$ and $B$ of a metric spaces $X$, we put 
$$
|AB|=\inf\bigl\{|ab|:a\in A,\,b\in B\bigr\}\ \ \text{and}\ \ |AB|'=\sup\bigl\{|ab|:a\in A,\,b\in B\bigr\}.
$$
\end{notation}

\section{Local Structure of Gromov--Hausdorff Space near General Position Metric Spaces}
\markright{\thesection.~Local Structure of Gromov--Hausdorff Space}

In the present section we define the spaces in general position and we show that the balls of equal sufficiently small radius and centered at these spaces as isometric. \textbf{Starting from this place, we always suppose that for the finite metric spaces in consideration it is fixed the order of their points}. For convenience, the sets of points of these spaces we denote by $\{1,\ldots,n\}$.

\begin{dfn}
We say that a finite metric space $M$ \emph{is in general position}, or \emph{is a space of general position}, if all its nonzero distances are distinct, and all triangle inequalities are strict.
\end{dfn}

For a metric space $X$, we define the value $\dl(X)$ as the minimum of the following three numbers (here $\inf\0=\infty$):
\begin{flalign*}
\indent&s(X)=\inf\bigl\{|xy|:x\ne y\bigr\},&\\
\indent&e(X)=\inf\Bigl\{\bigl||xy|-|zw|\bigr|:x\ne y,\,z\ne w,\,\{x,y\}\ne\{z,w\}\Bigr\},&\\
\indent&t(X)=\inf\bigl\{|xy|+|yz|-|xz|:x\ne y\ne z\ne x\bigr\}.&
\end{flalign*}
By agreement $\inf\0=\infty$, for $X\in\cM$ each of the inequalities $e(X)>0$, $s(X)>0$, $t(X)>0$, $\dl(X)>0$ implies that the space $X$ is finite. Also, for a finite $X$,
\begin{enumerate}
\item $s(X)>0$ holds always, and $s(X)=\infty$ iff $X$ is one-point metric space;
\item $e(X)>0$ is equivalent to that all nonzero distances in $X$ are distinct; also, $e(X)=\infty$ iff $X$ consists of one or two points;
\item $t(X)>0$ is equivalent to that for each triple of pairwise distinct points from $X$, all triangle inequalities are strict; also, $t(X)=\infty$ iff $X$ consists of one or two points;
\item $\dl(X)>0$ is equivalent to that $X$ is in general position.
\end{enumerate}

\subsection{Canonical Partition}

Metric spaces belonging to sufficiently small neighborhoods of spaces in general position have uniquely determined partitions which we describe in the present section.

\begin{prop}\label{prop:canon_part}
Let $M=\{1,\ldots,n\}$ be a metric space. Then for any $0<\e\le s(M)/2$ and each $X\in\cM$ such that $2d_{GH}(M,X)<\e$, there exists a unique up to numeration by points of $M$ partition $X=\sqcup_{i=1}^nX_i$ possessing the following properties\/\rom:
\begin{enumerate}
\item\label{prop:canon_part:1} $\diam X_i<\e$\rom;
\item\label{prop:canon_part:2} for any $i,j\in M$ and any $x\in X_i$, $x'\in X_j$ \(here the indices $i$ and $j$ may be equal to each other\/\) it holds $\bigl||xx'|-|ij|\bigr|<\e$.
\end{enumerate}
\end{prop}

\begin{proof}
Choose an arbitrary $R\in\cR_\opt(M,X)$ and put $X_i=R(i)$. If $X_i\cap X_j\ne\0$ for some $i\ne j$, then $\dis R\ge|ij|\ge s(M)$, therefore, $2d_{GH}(M,X)=\dis R>\e$, a contradiction. Thus, $\{X_i\}_{i=1}^n$ is a partition of $X$.

Notice that
$$
\dis R=\max_{i,j}\biggl\{\diam X_i,\,\dis\Bigl[\bigl(\{i\}\x X_i\bigr)\cup\bigl(\{j\}\x X_j\bigr)\Bigr],\,j\ne i\biggr\}<\e,
$$
therefore, $\diam X_i<\e$ for all $i$, and $\bigl||xx'|-|ij|\bigr|<\e$ for all $j\ne i$. If $i=j$, then $\bigl||xx'|-|ij|\bigr|=|xx'|\le\diam X_i<\e$, so the item~(\ref{prop:canon_part:2}) holds for all $i,j\in M$, even in the case $i=j$.

Let $X=\sqcup_{i=1}^nY_i$ be another partition of this type. Notice that no one $Y_i$ can intersect $X_j$ and $X_k$ simultaneously for $j\ne k$, because $\diam Y_i<\e$ and for each $x_j\in X_j$, $x_k\in X_k$ it holds $|x_jx_k|>|jk|-\e\ge s(M)-\e\ge\e$. Thus, each $Y_i$ intersects exactly one $X_j$. Similarly, each $X_j$ intersects exactly one $Y_i$. Since $\{X_i\}$ and $\{Y_i\}$ are both partitions of the set $X$, then for each $X_i$, $Y_j$, $X_i\cap Y_j\ne\0$, we have $Y_j=X_i$.
\end{proof}

\begin{rk}\label{rk:canon_part_small}
Proposition~\ref{prop:canon_part} describes one-points spaces $M=\D_1$ as well. In this case we have $s(M)=\infty$, so here we do not have upper restrictions on $\e$ (moreover, $\e$ can be equal to $\infty$). Since $2d_{GH}(M,X)<\e$, then $\diam X<\e$. The partition  $\{X_i\}$ consists of exactly one element $X_1=X$. The item~(\ref{prop:canon_part:1}) is satisfied as well by the reasons discussed above. The item~(\ref{prop:canon_part:2}) is regarded as valid because its condition $i\ne j$ cannot appear in this case.
\end{rk}

\begin{rk}\label{rk:canon_part}
Notice that the item~(\ref{prop:canon_part:2}) of Proposition~\ref{prop:canon_part} implies $\diam X_i\le\e$, that is weaker than the item~(\ref{prop:canon_part:1}).
\end{rk}

Let us present one more useful version of Proposition~\ref{prop:canon_part}.

\begin{prop}\label{prop:canon_part_variant}
Let $M=\{1,\ldots,n\}$ be a metric space. Then for any $0<\e\le s(M)/2$, any $X\in\cM$, $2d_{GH}(M,X)<\e$, and each $R\in\cR_\opt(M,X)$ the family $\{R(i)\}_{i=1}^n$ is a partition of the set $X$ satisfying the following properties\/\rom:
\begin{enumerate}
\item $\diam R(i)<\e$\rom;
\item for any $i,j\in M$, $x\in R(i)$, $x'\in R(j)$ it holds $\bigl||xx'|-|ij|\bigr|<\e$.
\end{enumerate}
Moreover, if $R'$ is another optimal correspondence between $M$ and $X$, then the partitions $\{R(i)\}_{i=1}^n$ and $\{R'(i)\}_{i=1}^n$ may differ from each other only by numerations generated by the correspondences $i\mapsto R(i)$ and $i\mapsto R'(i)$.
\end{prop}

\begin{dfn}
The family $\{X_i\}$ from Proposition~\ref{prop:canon_part} we call the \emph{canonical partition of the space $X$ with respect to $M$}.
\end{dfn}

\begin{rk}\label{rk:nonuniqueness_canonic_part}
For $\#M=1$ the canonical partition of the space $X$ consists of exactly one element, and this element is the space $X$ itself. If $\#M=2$ and the elements of the canonical partition $X=A\sqcup B$ are nonisometric, then there always exist two different numerations of this partition by the points of the space $M$: $X_1=A$, $X_2=B$ and $X_1=B$, $X_2=A$. For spaces with $\#M>2$, $e(M)>0$, and for sufficiently small $\e$ such numerations are unique, see Proposition~\ref{prop:CanonPartGeneralSpace} below.
\end{rk}

\begin{prop}\label{prop:optimalRelationPartition}
Let $M=\{1,\ldots,n\}$ be a metric space. Choose an arbitrary $0<\e\le s(M)/4$, any $X,Y\in\cM$, $2d_{GH}(M,X)<\e$, $2d_{GH}(M,Y)<\e$, and let $\{X_i\}$, $\{Y_i\}$ denote the canonical partitions of $X$ and $Y$, respectively, w.r.t. $M$. Then for each $R\in\cR_\opt(X,Y)$ there exists a bijection $\psi\:M\to M$ such that for some $R_i\in\cR(X_i,Y_{\psi(i)})$ it holds $R=\sqcup_{i=1}^nR_i$.
\end{prop}

\begin{proof}
Let us estimate $\dis R$. To do that, consider $R'=\sqcup_{i=1}^nX_i\x Y_i$. Then, by Proposi\-tion~\ref{prop:canon_part}, we have
$$
\dis R\le\dis R'=\max\{\diam X_i,\,\diam Y_i,\,\bigl||X_iX_j|'-|Y_iY_j|\bigr|,\,\bigl||X_iX_j|-|Y_iY_j|'\bigr|,\,i\ne j\}\le2\e.
$$

Choose arbitrary $x,x'\in X_i$, $y\in R(x)$, $y'\in R(x')$, and show that $y$ and $y'$ belong to the same $Y_j$. Suppose otherwise, then for some $j\ne k$ we have $y\in Y_j$, $y'\in Y_k$, thus, by Proposition~\ref{prop:canon_part}, $|xx'|<\e$, $|yy'|>|jk|-\e\ge s(M)-\e\ge3\e$, so $\dis R\ge\bigl||xx'|-|yy'|\bigr|>2\e$, a contradiction.

Swapping $X$ and $Y$, we get that for any $(x,y),\,(x',y')\in R$ the points $x$, $x'$ belong to the same element of the canonical partition $\{X_i\}$ iff the points $y$, $y'$ belong to the same element of the canonical partition $\{Y_i\}$.
\end{proof}

\begin{prop}\label{prop:metric-change}
Let $M=\{1,\ldots,n\}$ and $N=\{1,\ldots,n\}$ be metric spaces. The distance between $i$ and $j$ in $M$ we denote by $|ij|_M$, and the distance between $i$ and $j$ in $N$ by $|ij|_N$. Suppose that $t(N)>0$, choose an arbitrary $0<\e\le\min\{s(M)/2,2s(N)/3,t(N)/3\}$, any $X\in\cM$, $2d_{GH}(M,X)<\e$, and let $\{X_i\}$ be the canonical partition of $X$ w.r.t. $M$. Define a function $\r$ on $X\x X$ as follows\/\rom: for $x\in X_i$ and $x'\in X_j$ we put $\r(x,x')=|xx'|-|ij|_M+|ij|_N$ \(here the indices $i$ and $j$ may be equal to each other\/\). Then $\r$ is a metric on $X$ coinciding with the initial metrics on each $X_i$. Moreover, if $V$ denotes the set $X$ with the metric $\r$, then $V\in\cM$ and $d_{GH}(N,V)\le d_{GH}(M,X)$.
\end{prop}

\begin{proof}
If $M=\D_1$, then $N=\D_1$, the function $\r$ coincides with the initial metric on $X_1=X$, so $V$ and $X$ are equal metric spaces, thus $d_{GH}(N,V)\le d_{GH}(M,X)$.

Now, let  $\#M\ge2$. Since, by Proposition~\ref{prop:canon_part}, for $i\ne j$, $x\in X_i$, $x'\in X_j$ it holds $|xx'|-|ij|_M+|ij|_N>-\e+s(N)\ge-\e+3\e/2>0$, then the function $\r$ is positively defined. It remains to verify that $\r$ satisfies the triangle inequality.

Let $x\in X_i$, $x'\in X_j$, $x''\in X_k$ be arbitrary points of the space $X$. We shall prove that $\r(x,x'')\le\r(x,x')+\r(x',x'')$, i.e., in explicit form,
\begin{equation}\label{eq:triangle-rule}
|xx'|+|x'x''|-|xx''|-|ij|_M-|jk|_M+|ik|_M+|ij|_N+|jk|_N-|ik|_N\ge0.
\end{equation}

Suppose first that $i=j=k$, then
\begin{multline*}
|xx'|+|x'x''|-|xx''|-|ij|_M-|jk|_M+|ik|_M+|ij|_N+|jk|_N-|ik|_N=\\ =|xx'|+|x'x''|-|xx''|\ge0.
\end{multline*}

Further, let exactly two from the three indices $i$, $j$, $k$ equal to each other. It suffices to analyze two cases: $i=j$ and $i=k$.

If $i=j$, then
\begin{multline*}
|xx'|+|x'x''|-|xx''|-|ij|_M-|jk|_M+|ik|_M+|ij|_N+|jk|_N-|ik|_N=\\ =|xx'|+|x'x''|-|xx''|-|ik|_M+|ik|_M+|ik|_N-|ik|_N=|xx'|+|x'x''|-|xx''|\ge0.
\end{multline*}

If $i=k$, then, taking into account that $x_k\in X_i$, we get
\begin{multline*}
|xx'|+|x'x''|-|xx''|-|ij|_M-|jk|_M+|ik|_M+|ij|_N+|jk|_N-|ik|_N=\\ =\bigl(|xx'|-|ij|_M\bigr)+\bigl(|x'x''|-|ij|_M\bigr)-|xx''|+2|ij|_N.
\end{multline*}
By Proposition~\ref{prop:canon_part}, we have $|xx'|-|ij|_M>-\e$, $|x'x''|-|ij|_M>-\e$, $-|xx''|\ge-\diam X_i>-\e$. On the other hand, $2|ij|_N\ge2s(N)\ge3\e$, that implies inequality~(\ref{eq:triangle-rule}).

At last, suppose that $n\ge 3$ and $i\ne j\ne k\ne i$. For convenience, let us denote the three chosen points of the space $X$ by $x_i\in X_i$, $x_j\in X_j$, and $x_k\in X_k$. By Proposition~\ref{prop:canon_part}, for each pair $\{p,q\}\ss\{i,j,k\}$ we have $|x_px_q|-|pq|_M>-\e$ and $-|x_px_p|+|pq|_M>-\e$, so
$$
|x_ix_j|+|x_jx_k|-|x_ix_k|-|ij|_M-|jk|_M+|ik|_M>-3\e.
$$
On the other hand, $|ij|_N+|jk|_N-|ik|_N\ge t(N)\ge3\e$, therefore, the inequality~(\ref{eq:triangle-rule}) holds.

Compactness of $V$ follows from compactness of the components $X_i$ (these components are compact because the distance between them is nonzero, thus, they are closed subsets of $X$).

Estimate $d_{GH}(N,V)$. Let $R\in\cR_\opt(M,X)$ be the same as in the proof of Proposition~\ref{prop:canon_part}, i.e., $X_i=R(i)$. Denote by $R_V$ the same correspondence $R$, but now considered as an element of $\cR(N,V)$. Let $(i,x),\,(j,x')\in R_V$, then for $i\ne j$ we have
$$
\bigl||ij|_N-\r(x,x')\bigr|=\bigl||ij|_N-|xx'|+|ij|_M-|ij|_N\bigr|=\bigl||ij|_M-|xx'|\bigr|,
$$
therefore, $2d_{GH}(N,V)\le\dis R_V=\dis R=2d_{GH}(M,X)$.
\end{proof}

The construction described in Proposition~\ref{prop:metric-change}, namely, the transformation of a metric space $X$ into a metric space $V$, will be used for realization of local isometry in the Gromov--Hausdorff space. However, word-by-word use of this transformation does not provide us with desirable result. The reason is that this construction depends on the numeration of the canonical partition elements by the elements of the space $M$, however, this numeration is not unique. One of the corresponding examples was described in Remark~\ref{rk:nonuniqueness_canonic_part}. Another example can be easily obtain if we take as $M$ a metric space with nontrivial symmetry group, for example, an $M$ such that all nonzero distances between its points are equal to each other.

Nevertheless, when $M$ consists of at most two points, this nonuniqueness is not essential. Thus, the next discussion has two parts. We start with the case when $M$ consists of one or two points, and we finish with the case $\#M\ge3$, where we impose on $M$ some additional restrictions like $e(M)>0$, which, as it was mentioned in Remark~\ref{rk:nonuniqueness_canonic_part}, guarantee uniqueness of the numeration of canonical partition elements (for sufficiently small $\e$).

\subsection{The Case of Spaces Consisting of One or Two Points}\label{subsec:small}

Let us start with $M=\D_1$. Then, under the notations from Proposition~\ref{prop:metric-change}, we have $N=\D_1=M$. By Remark~\ref{rk:canon_part_small}, there is no upper restrictions on $\e$, and the canonical partition of the space $X$ consists of just one element, namely, $X_1=X$. Then, the function $\r$ coincides with the initial metric on $X$, and, thus, the metric space $V$ is equal to $X$. Therefore, the mapping $X\to V$ is uniquely defined and is equal to the identity map of $U_\e(\D_1)$ onto itself. We denote this mapping by $D_{M,N,\e}$.

Now, suppose that $M=\{1,2\}$, $N=\{1,2\}$, and $0<\e\le\min\{s(M)/2,2s(N)/3\}$. Since $t(N)=\infty$, then the restrictions on $\e$ are exactly the same as in Proposition~\ref{prop:metric-change}. As it was mentioned in Remark~\ref{rk:nonuniqueness_canonic_part}, the canonical partition $\{A,B\}$ of the space $X$ has two numerations. However, for each of these numerations the space $V\in\cM$ from Proposition~\ref{prop:metric-change} is the same. Thus, we constructed a mapping $D_{M,N,\e}\:U_\e(M)\to U_\e(N)$.

Enforce the restrictions on $\e$ in such a way that Proposition~\ref{prop:optimalRelationPartition} can be useful, namely, let us demand that $0<\e\le\min\{s(M)/4,2s(N)/3\}$. Choose arbitrary $X,Y\in U_\e(M)$ and put $V=D_{M,N,\e}(X)$, $W=D_{M,N,\e}(Y)$. Let $\{X_1,X_2\}$ and $\{Y_1,Y_2\}$ be the canonical partitions of $X$ and $Y$, respectively, w.r.t. $M$, and let $R\in\cR_\opt(X,Y)$. By Proposition~\ref{prop:optimalRelationPartition}, there exists a bijection $\psi\:M\to M$ such that for some $R_i\in\cR(X_i,Y_{\psi(i)})$ it holds $R=R_1\sqcup R_2$. Therefore,
$$
2d_{GH}(X,Y)=\max\{\dis R_1,\,\dis R_2,\,\bigl||x_1x_2|-|y_1y_2|\bigr|:(x_1,y_1)\in R_1,\,(x_2,y_2)\in R_2\}.
$$
Since $V$ and $W$ differ from $X$ and $Y$ by the distances functions only, the $R$ can be also considered as an element from $\cR(V,W)$. Since the transition to $V$ and $W$ preserves the metrics on $X_i$ and $Y_i$, then all $\dis R_i$ remain fixed. Besides that, the distances $|x_1x_2|$ and $|y_1y_2|$ were changed by the same value $|12|_N-|12|_M$, so the values $\bigl||x_1x_2|-|y_1y_2|\bigr|$ remain fixed as well. This implies that for the $R$ considered as an element from $\cR(V,W)$, its distortion is the same as of the initial $R$, therefore, $d_{GH}(V,W)\le d_{GH}(X,Y)$.

Now, let us ``symmetrize'' the condition on $\e$, i.e., we demand that
$$
0<\e\le\min\{s(M)/4,s(N)/4\}.
$$
Then we get well-defined mapping $D_{N,M,\e}\:U_\e(N)\to U_\e(M)$ which is obviously coincides with the inverse to $D_{M,N,\e}$. Therefore, $d_{GH}(X,Y)\le d_{GH}(V,W)$, thus, we have proven the following result.

\begin{prop}\label{prop:local_isometry_gen_pos_small}
Let $M=\{1,\ldots,n\}$ and $N=\{1,\ldots,n\}$ be metric spaces, $n\le2$, and $0<\e\le\min\{s(M)/4,e(M)/4\}$. Then the mapping $D_{M,N,\e}\:U_\e(M)\to U_\e(N)$ is an isometry.
\end{prop}

\subsection{The Case of Spaces Consisting of Three or More Points}

The next proposition guarantees, under special assumptions, uniqueness of numeration of canonical partition.

\begin{prop}\label{prop:CanonPartGeneralSpace}
Let $M=\{1,\ldots,n\}$ be a metric space such that $n\ge3$ and $e(M)>0$. Choose an arbitrary $0<\e\le\frac12\min\{s(M),e(M)\}$, any $X\in\cM$, $2d_{GH}(M,X)<\e$, and let $\{X_i\}$ and $\{X'_i\}$ be the canonical partitions of $X$ w.r.t. $M$. Then $X_i=X'_i$ for each $i$.
\end{prop}

\begin{proof}
By Proposition~\ref{prop:canon_part}, the partitions $\{X_i\}$ and $\{X'_i\}$ differ from each other by a numeration, i.e., there exists a permutation $\v\:M\to M$ such that $X_i=X'_{\v(i)}$. We have to show that the permutation $\v$ is identical.

Suppose otherwise, i.e., that for some $j$ it holds $X_j=X'_i$, $i\ne j$. Since $n>0$, then there exists $k\not\in\{i,j\}$. Put $X'_p=X_k$. Since $k\ne j$, then $p\ne i$. Besides that, $i\ne j$ and $i\ne k$, thus $\{j,k\}$ and $\{i,p\}$ are distinct pairs, therefore, $\bigl||jk|-|ip|\bigr|\ge e(M)\ge\min\bigl\{s(M),e(M)\bigr\}\ge2\e$. But for any $x_j\in X_j=X'_i$ and $x_k\in X_k=X'_p$ it holds $\bigl||x_jx_k|-|jk|\bigr|<\e$ and $\bigl||x_jx_k|-|ip|\bigr|<\e$, whence it follows that
$$
\bigl||jk|-|ip|\bigr|=\bigl||jk|-|x_jx_k|+|x_jx_k|-|ip|\bigr|\le \bigl||x_jx_k|-|jk|\bigr|+\bigl||x_jx_k|-|ip|\bigr|<2\e,
$$
a contradiction.
\end{proof}

The next result can be proved similarly with Proposition~\ref{prop:CanonPartGeneralSpace}.

\begin{prop}\label{prop:correspondence-partition}
Let $M=\{1,\ldots,n\}$ be a metric space, $n\ge 3$, $e(M)>0$. Choose an arbitrary $0<\e\le\frac14\min\bigl\{s(M),\,e(M)\bigr\}$, any $X,Y\in\cM$, $2d_{GH}(M,X)<\e$, $2d_{GH}(M,Y)<\e$, and let $\{X_i\}$ and $\{Y_i\}$ denote the canonical partitions of $X$ and $Y$, respectively, w.r.t. $M$. Then for each $R\in\cR_\opt(X,Y)$ there exist $R_i\in\cR(X_i,Y_i)$ such that $R=\sqcup_{i=1}^nR_i$.
\end{prop}

\begin{proof}
By Proposition~\ref{prop:optimalRelationPartition}, there exist a bijection $\psi\:M\to M$ and $R_i\in\cR(X_i,Y_{\psi(i)})$ such that $R=\sqcup_{i=1}^nR_i$. We have to prove that $\psi$ is identical.

Suppose otherwise, i.e., that for some $j$ it holds $\psi(j)=i$, $i\ne j$. Since $n>0$, then there exists $k\not\in\{i,j\}$. Put $p=\psi(k)$. Since $k\ne j$, then $p\ne i$. Besides that, $i\ne j$ and $i\ne k$, so $\{j,k\}$ and $\{i,p\}$ are distinct pairs, therefore, $\bigl||jk|-|ip|\bigr|\ge s(M)\ge\min\bigl\{s(M),e(M)\bigr\}\ge4\e$. But for any $x_j\in X_j$, $x_k\in X_k$,  $y_i\in R(x_j)\ss Y_i$, $y_p\in R(x_k)\ss Y_p$ we have $\bigl||y_iy_p|-|ip|\bigr|<\e$, $\bigl||x_jx_k|-|jk|\bigr|<\e$, $\bigl||x_jx_k|-|y_iy_p|\bigr|<2\e$, where the latter inequality holds because $(x_j,y_i),\,(x_k,y_p)\in R$ and $\dis R=2d_{GH}(X,Y)\le2d_{GH}(X,M)+2d_{GH}(M,Y)<2\e$. Thus, we get the following estimation:
\begin{multline*}
\bigl||jk|-|ip|\bigr|=\bigl||jk|-|x_jx_k|+|x_jx_k|-|y_iy_p|+|y_iy_p|-|ip|\bigr|\le\\ \le \bigl||jk|-|x_jx_k|\bigr|+\bigl||x_jx_k|-|y_iy_p|\bigr|+\bigl||y_iy_p|-|ip|\bigr|<4\e,
\end{multline*}
a contradiction.
\end{proof}

\begin{prop}\label{prop:unique-map}
Under the notations of Proposition~$\ref{prop:metric-change}$, suppose additionally that $e(M)>0$ and $0<\e\le\min\{s(M)/2,e(M)/2,2s(N)/3,t(N)/3\}$, then the metric $\r$ and the space $V$ are uniquely defined.
\end{prop}

\begin{proof}
This follows from Proposition~\ref{prop:CanonPartGeneralSpace} which guarantees that the canonical partition of the space $X$ is uniquely defined.
\end{proof}

\subsection{The Main Results}

\begin{notation}
In Proposition~\ref{prop:unique-map}, for the spaces $M=\{1,\ldots,n\}$, $N=\{1,\ldots,n\}$, $n\ge 3$, and sufficiently small $\e>0$, we have constructed a mapping from $U_\e(M)\ss\cM$ to $U_\e(N)\ss\cM$ which takes each space $X\in U_\e(M)$ to the space obtained from $X$ by the following replace of its metric: if $\{X_i\}$ is the canonical partition of the space $X$, and $x\in X_i$, $x'\in X_j$, then the distance $|xx'|$ increases by $|ij|_N-|ij|_M$. Similarly with the case $n\le2$, we denote the obtained mapping by $D_{M,N,\e}$.
\end{notation}

\begin{thm}\label{thm:local_isometry_gen_pos}
Let $M=\{1,\ldots,n\}$ and $N=\{1,\ldots,n\}$ be general position metric spaces, and $0<\e\le\min\{s(M)/4,e(M)/4,t(M)/3,s(N)/4,e(N)/4,t(N)/3\}$. Then the mapping $D_{M,N,\e}\:U_\e(M)\to U_\e(N)$ is an isometry.
\end{thm}

\begin{proof}
Notice that for $n\le 2$ the condition of the theorem coincides with the one of Proposition~\ref{prop:local_isometry_gen_pos_small}, so our theorem is proved in this case. Now, let $n\ge3$. The subsequent reasoning is very similar with the one from subsection~\ref{subsec:small}.

Choose arbitrary $X,\,Y\in U_\e(M)$ and put $V=D_{M,N,\e}(X)$, $W=D_{M,N,\e}(Y)$. Let $\{X_i\}$ and $\{Y_i\}$ be the canonical partitions of $X$ and $Y$, respectively, w.r.t. $M$, and let $R\in\cR_\opt(X,Y)$. By Proposition~\ref{prop:correspondence-partition}, there exist $R_i\in\cR(X_i,Y_i)$ such that $R=\sqcup R_i$. Therefore,
$$
2d_{GH}(X,Y)=\max_{i,j,k}\{\dis R_k,\,\bigl||x_ix_j|-|y_iy_j|\bigr|:(x_i,y_i)\in R_i,\,(x_j,y_j)\in R_j,\,i\ne j\}.
$$
Since $V$ and $W$ differ from $X$ and $Y$ by the distance functions only, then $R$ can be considered as an element from $\cR(V,W)$. Since the transition to $V$ and $W$ preserves the metrics of $X_i$ and $Y_i$, the $\dis R_i$ remains fixed as well. Besides that, for each $i\ne j$ the distances $|x_ix_j|$ and $|y_iy_j|$ were changed by the same value $|ij|_N-|ij|_M$, so the values $\bigl||x_ix_j|-|y_iy_j|\bigr|$ remain fixed. This implies that for the $R$ considered as an element from $\cR(V,W)$, its distortion is the same as of the initial $R$, therefore, $d_{GH}(V,W)\le d_{GH}(X,Y)$.

It remains to notice that, by the symmetry of assertions imposed on $M$, $N$, and $\e$, all the above constructions can be reversed, that generates the mappings $D_{N,M,\e}\:U_\e(N)\to U_\e(M)$ inverse to $D_{M,N,\e}$. Thus, $d_{GH}(X,Y)\le d_{GH}(V,W)$, that completes the proof.
\end{proof}

Notice that for $n\ge3$ the mapping $D_{M,N,\e}$ depends essentially on the numerations of the points from the spaces $M$ and $N$. In particulary, if we take as $N$ the space $M$ with some other numerations of its points, then we obtain nontrivial and distinct from each other isometries of $U_\e(M)$ onto itself. Indeed, if we choose in $U_\e(M)$ a space $X$ such that the components of its canonical partition consist of different numbers of points, say $\#X_i=i$, then the images of such $X$ under the mappings $D_{M,N,\e}$ are not isometric to each other because no one isometry of $X$ can nontrivially permute the $\{X_i\}$.

Let $S_n$ be the permutation group of the set $\{1,\ldots,n\}$, then for each $\t\in S_n$ we denote by $M^\t$ the metric space $M$ for which its point $i$ has the number $\t(i)$. The above discussion leads to the following conclusion.

\begin{cor}
Let $M=\{1,\ldots,n\}$, $n\ge3$, be a metric space in general position. Then for sufficiently small $\e>0$ and any $\t,\s\in S_n$, $\t\ne\s$, the isometries $D_{M,M^\t,\e}$ and $D_{M,M^\s,\e}$ of the ball $U_\e(M)$ are distinct. Thus, the isometry group of the round neighborhood $U_\e(M)$ contains a subgroup isomorphic to the permutation group $S_n$.
\end{cor}

Since, by Proposition~\ref{prop:elem_prop}, multiplication by $\l>0$ of all distances of a compact metric space is a homothety of the Gromov--Hausdorff space, centered at one-point metric space $\D_1$, and the restrictions imposed on $\e$ in Theorem~\ref{thm:local_isometry_gen_pos} are $1$-homogeneous by $\l$, then each isometry $D_{M,N,\e}\:U_\e(M)\to U_\e(N)$ generates an isometry $D_{\l M,\l N,\l\e}\:U_{\l\e}(\l M)\to U_{\l\e}(\l N)$ and, as a corollary, the one of the cones $CU_\e(M)$ and $CU_\e(N)$ (on $\D_1$ we extend this mapping by continuity).

\begin{cor}
Let $M=\{1,\ldots,n\}$ and $N=\{1,\ldots,n\}$ be metric spaces of general position, $n\ge3$, then for sufficiently small $\e>0$ there exists an isometry of the cones $CU_\e(M)$ and $CU_\e(N)$ fixed at $\D_1$.
\end{cor}

\end{document}